\documentclass[12pt, reqno]{amsart}
\usepackage{amssymb,dsfont}
\usepackage{eucal}
\usepackage{amsmath}
\usepackage{amscd}
\usepackage[dvips]{color}
\usepackage{multicol}
\usepackage[all]{xy}           
\usepackage{graphicx}
\usepackage{color}
\usepackage{colordvi}
\usepackage{xspace}
\usepackage{txfonts}
\usepackage{lscape}
\usepackage{tikz} 

\numberwithin{equation}{section}
\usepackage[shortlabels]{enumitem}
\usepackage{ifpdf}
\ifpdf
\usepackage[colorlinks,final,backref=page, hyperindex]{hyperref}
\else
\fi

\newtheorem{theorem}{Theorem}[section]
\newtheorem{definition}[theorem]{Definition}
\newtheorem{lemma}[theorem]{Lemma}
\newtheorem{proposition}[theorem]{Proposition}
\newtheorem{corollary}[theorem]{Corollary}
\newtheorem{case}{Case}
\newtheorem{claim}{Claim}
\numberwithin{equation}{section}

\def\bfw{\mathbf{w}}

\def\al{\alpha}
\def\be{\beta}

\def\hh{\mathfrak{h}}
\def\gg{\mathfrak{g}}
\def \<{\langle}
\def \>{\rangle}

\def\GG{\mathcal{G}}

\def\UU{\mathcal{U}}

\def\LL{\mathcal{L}}

\def\WW{\mathcal{W}}

\newcommand{\C}{\mathbb C}

\newcommand{\N}{\mathbb{N}}
\newcommand{\Z}{\mathbb{Z}}

\def\Supp{\mathrm{Supp}}
\allowdisplaybreaks

\begin{document}
\title[planar Galilean conformal algebra]{Tensor product modules over the planar Galilean conformal algebra from free modules of rank one}
\author{Jin Cheng}
\address{J. Cheng: School of Mathematics and Statistics, Shandong Normal University,
	Jinan, Shandong 250014, P. R. China}
\email{cheng934@mail.ustc.edu.cn}
\author{Dongfang Gao}
  \address{D. Gao: Chern Institute of Mathematics and LPMC, Nankai University, Tianjin 300071, P. R. China, and
 Institut Camille Jordan, Universit\'{e} Claude Bernard Lyon 1, Lyon, 69622, France}
  \email{gao@math.univ-lyon1.fr}
 \author{Ziting Zeng}
\address{Z. Zeng: Laboratory of Mathematics and Complex Systems,  School of Mathematical Science, Beijing Normal University, Beijing 100875, P. R. China}
\email{zengzt@bnu.edu.cn}

\maketitle

\begin{abstract}
In this paper, we investigate the irreducible tensor product modules over the planar Galilean conformal algebra $\GG$ named by Aizawa,
which is the infinite-dimensional Galilean conformal algebra introduced by Bagchi-Gopakumar in $(2+1)$ dimensional space-time.
We give the necessary and sufficient conditions  for the tensor product modules of any two of $\UU(\hh)$-free modules of rank one over $\GG$ to be irreducible,
where $\hh$ is the Cartan subalgebra of $\GG$.
 Furthermore, the isomorphism classes of these irreducible tensor product modules are determined. 
 As an application, we obtain the necessary conditions for the tensor product modules of any two of $\UU(\C L_0)$-free modules of rank one over Witt algebra and Heisenberg-Virasoro algebra to be irreducible.
\end{abstract}

\raggedbottom
\section{Introduction}

Infinite-dimensional Galilean conformal algebras were introduced by physicists Bagchi and Gopakumar (see \cite{BG}) when they studied the AdS/CFT conjecture, which was introduced by Maldacena in \cite{Ma}.
Moreover, those algebras appear in the context of Galilean electrodynamics (see \cite{BBM, FHHO}) and are related to Navier-Stokes equations (see \cite{BMW}).
Recall that the infinite-dimensional Galilean conformal algebra in $(d+1)$ dimensional space-time is spanned by $$\{L^{(m)}, J_{ij}^{(m)}, M_i^{(m)}~|~m\in\Z, i,j=1,2,\cdots,d\},$$
where
\begin{align*}
&L^{(m)}=-(m+1)t^m\sum_{i=1}^dx_i\partial_i-t^{m+1}\partial_t,\\
&J_{ij}^{(m)}=-t^m(x_i\partial_j-x_j\partial_i),\\
&M_i^{(m)}=t^{m+1}\partial_i.
\end{align*}
The Lie brackets are defined as follows
\begin{align}
&[L^{(m)},L^{(n)}]=(m-n)L^{(m+n)},\notag\\
&[J_{ij}^{(m)}, J_{rs}^{(n)}]=\delta_{ir}J_{js}^{(m+n)}+\delta_{is}J_{rj}^{(m+n)}+\delta_{jr}J_{si}^{(m+n)}+\delta_{js}J_{ir}^{(m+n)},\label{IGCA}\\
&[L^{(m)},J_{ij}^{(n)}]=-nJ_{ij}^{(m+n)},\qquad [L^{(m)},M_i^{(n)}]=(m-n)M_i^{(m+n)},\notag\\
&[J_{ij}^{(m)},M_k^{(n)}]=-(\delta_{jk}M_i^{(m+n)}-\delta_{ik}M_j^{(m+n)}),\quad [M_i^{(m)},M_j^{(n)}]=0,\notag
\end{align}
where $m,n\in\Z, 1\leq i,j,k, r,s\leq d$.
Note that the subalgebra spanned by $$\{L^{(m)}, J_{ij}^{(0)}, M_i^{(m)}~|~m=0,\pm1, i,j=1,2,\cdots,d\},$$
is exactly the finite-dimensional Galilean conformal algebra in $(d+1)$ dimensional space-time. 
So infinite-dimensional Galilean conformal algebra is a very natural infinite-dimensional extension of finite-dimensional Galilean conformal algebra (see \cite{BG} for more details).
Moreover, from Eq.~\eqref{IGCA} we see that infinite-dimensional Galilean conformal algebras are related to many well-known infinite-dimensional Lie algebras, 
for example, the affine Lie algebras (see \cite{Car, Kac}), the Virasoro algebra (see \cite{Vir}), the Heisenberg-Virasoro algebra (see \cite{ADKP}), the $W$-algebra $W(2,2)$ (see \cite{ZD}) and so on.
Therefore infinite-dimensional Galilean conformal algebras are really important in mathematics and mathematical physics rather than artificial. 
In this paper, we mainly concern the representation theory of infinite-dimensional Galilean conformal algebra in $(2+1)$ dimensional space-time, which was named the planar Galilean conformal algebra by Aizawa in \cite{A}. This algebra is also the special case of \cite{MT}. 

Weight modules have been the popular for many Lie algebras with the triangular decompositions.
In particular, to some extent, Harish-Chandra modules (weight modules with finite-dimensional weight spaces) are well understood for many infinite-dimensional Lie algebras, for example, 
the Virasoro algebra (see \cite{As, LZ, M, S}), the Heisenberg-Virasoro algebra (see \cite{LG, LZ1}), 
the affine Kac-Moody Lie algebras (see \cite{CP, GZ}) and the Witt algebra of rank $n$ (see \cite{BF}).
There are also some results about weight modules with infinite-dimensional weight spaces (see \cite{BBFK, ChenGZ, GaoZ, MZ}).

For the last two decades, two families of non-weight modules attract more attentions from mathematicians,
called Whittaker modules and $\UU(\hh)$-free modules respectively, 
where $\hh$ is the Cartan subalgebra of the Lie algebra.
Whittaker modules were introduced by Kostant (see \cite{K}) for finite-dimensional complex semisimple Lie algebras. Actually, these modules for $\mathrm{sl_2(\C)}$ were earlier constructed by Arnal and Pinczon in \cite{AP}. 
The notation of $\UU(\hh)$-free modules was first introduced by Nilsson in \cite{Nil} for the simple Lie
algebra $\mathrm{sl_{n+1}(\C)}$. At the same time, these modules were introduced in a very different approach in
\cite{TZ}.

For the planar Galilean conformal algebra $\GG$,
Verma modules and coadjoint representations were studied in \cite{A},
Harish-Chandra modules can be deduced from \cite{CLW}, 
 Whittaker modules were investigated in \cite{CYY, G}, restricted modules under the certain conditions were characterized in \cite{CY, GG}, $\UU(\hh)$-free modules of rank one over $\GG$ were completely classified in \cite{CGZ}.
In order to obtain new irreducible $\GG$-modules, one of the methods is to consider the tensor product modules of known modules. 
But it is generally not easy to determine the conditions for the tensor product modules to be irreducible.
In particular, as far as we know, 
there are few of conclusions about the tensor products of $\UU(\hh)$-free modules for any infinite-dimensional Lie algebras.
However, in the present paper, we give some complete results about the tensor product modules of any two of $\UU(\hh)$-free modules of rank one over $\GG$.

The paper is organized as follows.
In Section 2, we recall notations related to the planar Galilean conformal algebra and
collect some results about $\UU(\hh)$-free modules over $\GG$ for later use.
In particular, we review that there are three families of $\UU(\hh)$-free modules of rank one over $\GG$,
denoted by $\Omega(\lambda,\eta,\sigma,0), \Omega(\lambda,\eta,0,\sigma)$ and $\Omega(\lambda,\delta,0,0)$
respectively (see Lemma \ref{three-families}).
In Section 3, we give the necessary and sufficient conditions for the tensor product module $\Omega(\lambda_1,\eta_1,\sigma_1,0)\otimes\Omega(\lambda_2,\eta_2,0,\sigma_2)$ to be irreducible, see Theorem \ref{irr-12}.
Also, the isomorphism classes of these irreducible tensor product modules are determined, see Theorem \ref{iso-12}.
In Section 4, the irreducible criteria and isomorphism classes of the tensor product modules $\Omega(\lambda_1,\eta_1,\sigma_1,0)\otimes\Omega(\lambda_2,\eta_2,\sigma_2,0)$ and $\Omega(\lambda_1,\eta_1,0,\sigma_1)\otimes\Omega(\lambda_2,\eta_2,0,\sigma_2)$ are obtained, see Theorems \ref{irr-11}, \ref{irr-22}, \ref{iso-11} and \ref{iso-22}.
In Section 5, we deduce some results about tensor product modules over the Witt algebra and the Heisenberg-Virasoro algebra.

Throughout this paper, we denote by $\Z, \Z_+, \N, \C$ and $\C^*$ the sets of integers, positive integers,  non-negative integers, complex numbers and nonzero complex numbers respectively.
All vector spaces and algebras are over $\C$.
We denote by $\UU(\gg)$ the universal enveloping algebra of a Lie algebra $\gg$.

\section{Notations and preliminaries}
In this section, we define some notations and collect some known results about free modules of rank one over $\GG$.

By Eq.~\eqref{IGCA} we see that infinite-dimensional Galilean conformal algebra in $(2+1)$ dimensional space-time is spanned by 
$$\{L^{(m)}, J_{12}^{(m)}, M_1^{(m)}, M_2^{(m)}~|~m\in\Z\},$$
satisfying the following Lie brackets
\begin{align*}
&[L^{(m)},L^{(n)}]=(m-n)L^{(m+n)},\qquad [L^{(m)},J_{12}^{(n)}]=-nJ_{12}^{(m+n)},\\
&[L^{(m)},M_1^{(n)}]=(m-n)M_1^{(m+n)},\quad\ \ [L^{(m)},M_2^{(n)}]=(m-n)M_2^{(m+n)},\\
&[J_{12}^{(m)},M_1^{(n)}]=M_2^{(m+n)},\qquad\qquad\ \  [J_{12}^{(m)},M_2^{(n)}]=-M_1^{(m+n)}, \\
&[J_{12}^{(m)}, J_{12}^{(n)}]=[M_1^{(m)},M_1^{(n)}]=[M_2^{(m)},M_2^{(n)}]=[M_1^{(m)},M_2^{(n)}]=0,\ \ \ {\rm for \ all \ }\  m,n\in\Z.
\end{align*}
For convenience, we would like to simplify the notations (see \cite{CYY}). Let
$$L_n=-L^{(n)},\ \ H_n=\sqrt{-1}J_{12}^{(n)},\ \ I_n=M_1^{(n)}+\sqrt{-1}M_2^{(n)},\ \ J_n=M_1^{(n)}-\sqrt{-1}M_2^{(n)}, {\rm \ \ for\ \ all\ \ } n\in\Z.$$
We now can describe the definition of the planar Galilean conformal algebra as follows.
\begin{definition}\label{Galilean}
	The {\bf planar Galilean conformal algebra} $\GG$ is an infinite-dimensional Lie algebra with a basis 
	$\{L_m, H_m,I_m,J_m~|~m\in\Z\}$ subject to the following commutation relations
\begin{equation}\label{planar Galilean conformal algebra}
\begin{split}
&[L_m, L_n]=(n-m)L_{m+n},\quad [L_m, H_n]=nH_{m+n}, \quad [L_m, I_n]=(n-m)I_{m+n},\\
&[L_m, J_n]=(n-m)J_{m+n},\quad [H_m,I_n]=I_{m+n},\qquad \  [H_m,J_n]=-J_{m+n},\\
&[H_m, H_n]=[I_m, I_n]=[J_m, J_n]=[I_m, J_n]=0, \quad {\rm \ \ for\ \ all\ \ } m,n\in\Z.
\end{split}
\end{equation}
\end{definition}

It is clear that $\GG$ contains the several important subalgebras.
\begin{enumerate}
	\item Witt algebra: the subalgebra $\WW$ spanned by $\{L_m~|~m\in\Z\}$ is the Witt algebra,
	which is the derivation Lie algebra of Laurent polynomial algebra of one variable.
	\item Heisenberg-Virasoro algebra: the subalgebra $\LL$ spanned by $\{L_m, H_m~|~m\in\Z\}$ is the Heisenberg-Virasoro algebra,
	which was introduced by E. Arbarello, C. De Concini, V. G. Kac and C. Procesi in \cite{ADKP}, where the authors established a canonical isomorphism between the second cohomology of the certain Lie algebra
	and the second singular cohomology of the certain moduli space.
	\item $W$-algebra $W(2,2)$: the subalgebra $W(2,2)$ spanned by $\{L_m, I_m~|~m\in\Z\}$ or $\{L_m, J_m~|~m\in\Z\}$ is the centerless $W$-algebra $W(2,2)$, which was introduced by Zhang and Dong in order to classify the certain simple vertex operator algebras (see \cite{ZD}). 
\end{enumerate}

\begin{definition}
	Let $M$ be a $\GG$-module. Then $M$ is called a $\UU(\hh)$-{\bf free module of rank one} over $\GG$ if $M$ is free of rank one as $\UU(\hh)$-module, where $\hh=\C L_0+\C H_0$ is the Cartan subalgebra of $\GG$.
\end{definition}

From \cite{CGZ} we have the following lemma.
\begin{lemma}\label{three-families}
	\begin{enumerate}[$(1)$]
		\setlength{\itemindent}{-1.6em}\item Let $\lambda\in\C^*,\eta\in\C,\sigma\in\C[X]$ with $\sigma\ne 0$. 
		Then the polynomial algebra $\C[X,Y]$ is a $\GG$-module with the following actions
		\begin{align*}\label{one}
		&L_m(f(X,Y))=\lambda^{m}(Y-mX+m\eta)f(X,Y-m),\notag\\
		&H_m(f(X,Y))=\lambda^mXf(X,Y-m),\\
		&I_m(f(X,Y))=\lambda^m\sigma f(X-1,Y-m),\notag\\
		&J_m(f(X,Y))=0,\ \ {\rm \ for\ \ all\ \ } m\in\Z,\  f(X,Y)\in \C[X,Y].\notag
		\end{align*}
	    This module is denoted by $\Omega(\lambda,\eta,\sigma,0)$.
		\item Let $\lambda\in\C^*, \eta\in\C, \sigma\in\C[S]$ with $\sigma\ne 0$. 
		Then the polynomial algebra $\C[S,T]$ is endowed with a $\GG$-module structure by the following actions
		\begin{equation}\label{two}
		\begin{split}
		&L_m(f(S,T))=\lambda^{m}(T+mS+m\eta)f(S,T-m),\\
		&H_m(f(S,T))=\lambda^mSf(S,T-m),\\
		&I_m(f(S,T))=0,\\
		&J_m(f(S,T))=\lambda^m\sigma f(S+1,T-m), \ \ {\rm \ for\ \ all\ \ }m\in\Z,\ f(S,T)\in \C[S,T]. \notag
		\end{split}
		\end{equation}
		This module is denoted by $\Omega(\lambda,\eta,0,\sigma)$.
		\item Let $\lambda\in\C^*,\delta\in\C[P]$. 
		Then the polynomial algebra $\C[P,Q]$ has a $\GG$-module structure via the following actions
		\begin{equation*}\label{one}
		\begin{split}
		&L_m(f(P,Q))=\lambda^{m}(Q+m\delta)f(P,Q-m),\\
		&H_m(f(P,Q))=\lambda^mPf(P,Q-m),\\
		&I_m(f(P,Q))=J_m(f(P,Q))=0, \ \ {\rm \ for\ \ all\ \ }m\in\Z,\ f(P,Q)\in \C[P,Q].
		\end{split}
		\end{equation*}
	 We denote this module by $\Omega(\lambda,\delta,0,0)$.
	\end{enumerate}
\end{lemma}

By (\cite[Theorem 4.12]{CGZ}) we know that 
the above three families of modules exhaust all $\UU(\hh)$-free modules of rank one over $\GG$ up to isomorphism.
Moreover, the following lemma holds by Theorems 4.13, 4.14 and 4.15 in \cite{CGZ}.
\begin{lemma}\label{U(h)-irr}
	Let $\lambda\in\C^*,\eta\in\C,\sigma\in\C[X], \sigma'\in\C[S]$ and $\delta\in\C[P]$ with $\sigma\ne 0, \sigma'\ne 0$. Then
	\begin{enumerate}[$(1)$]
		\item $\Omega(\lambda,\eta,\sigma,0)$ is an irreducible $\GG$-module if and only if $\sigma$ is a nonzero constant;
		\item $\Omega(\lambda,\eta,0,\sigma')$ is an irreducible $\GG$-module if and only if $\sigma'$ is a nonzero constant;
		\item $\Omega(\lambda,\delta,0,0)$ is always reducible as $\GG$-module.
	\end{enumerate}
\end{lemma}

Therefore, in this paper, we mainly study the irreducibility of tensor product $\GG$-modules 
\begin{align*}
&\Omega(\lambda_1,\eta_1,\sigma_1,0)\otimes \Omega(\lambda_2,\eta_2,0,\sigma_2), \\
&\Omega(\lambda_1,\eta_1,\sigma_1,0)\otimes \Omega(\lambda_2,\eta_2,\sigma_2,0)
\end{align*}
and
\begin{align*}
\Omega(\lambda_1,\eta_1,0,\sigma_1)\otimes \Omega(\lambda_2,\eta_2,0,\sigma_2),
\end{align*}
where $\lambda_1,\lambda_2,\sigma_1,\sigma_2\in\C^*, \eta_1,\eta_2\in\C$,
since $\Omega(\lambda,\delta,0,0)$ is always a reducible $\GG$-module for any $\lambda\in\C^*, \delta\in\C[P]$.

Now we conclude this section by defining a total order on $\N^4$,
where $\N^4$ denotes the set of all vectors of the form $\bar{\al}:=(\al_1,\al_2,\al_3,\al_4)$ with entries in $\N$. 

For any $\bar{\al}\in\N^4$, denote
\begin{equation*}
\bfw(\bar{\al})=\al_1+\al_2+\al_3+\al_4,
\end{equation*}
which is a nonnegative integer.
Let $\bar{\al},\bar{\be}\in \N^4$. We say
\begin{align*}
\bar{\al} \succ \bar{\be},\quad {\rm\ if\ } & \bfw(\bar{\al})>\bfw(\bar{\be}),\\
{\rm\ or\ } &\bfw(\bar{\al})=\bfw(\bar{\be}), \al_4>\be_4,\\
{\rm\ or\ }& \bfw(\bar{\al})=\bfw(\bar{\be}), \al_4=\be_4, \al_3>\be_3,\\
{\rm\ or\ }& \bfw(\bar{\al})=\bfw(\bar{\be}), \al_4=\be_4, \al_3=\be_3, \al_2>\be_2,\\
{\rm\ or\ } &\bfw(\bar{\al})=\bfw(\bar{\be}), \al_4=\be_4, \al_3=\be_3, \al_2=\be_2, \al_1>\be_1.
\end{align*}
Obviously, the $\succ$ is a total order on $\N^4$.

\section{$\Omega(\lambda_1,\eta_1,\sigma_1,0)\otimes\Omega(\lambda_2,\eta_2,0,\sigma_2)$}\label{Sec3}
In this section, we study the tensor product $\GG$-module $\Omega(\lambda_1,\eta_1,\sigma_1,0)\otimes\Omega(\lambda_2,\eta_2,0,\sigma_2)$, where $\lambda_1,\lambda_2,\sigma_1,\sigma_2\in\C^*,\eta_1, \eta_2\in\C$.
More precisely, we give the necessary and sufficient condition for the $\GG$-module $\Omega(\lambda_1,\eta_1,\sigma_1,0)\otimes\Omega(\lambda_2,\eta_2,0,\sigma_2)$ to be irreducible.
Then the isomorphism classes of this family of irreducible tensor product $\GG$-modules are determined.

\subsection{Irreducibility}\label{sub3-1}
In this subsection, we determine the irreducibility of tensor product module $\Omega(\lambda_1,\eta_1,\sigma_1,0)\otimes\Omega(\lambda_2,\eta_2,0,\sigma_2)$ over $\GG$.
\begin{lemma}\label{nonzeroV}
	For any $\lambda_1, \lambda_2, \sigma_1, \sigma_2\in\C^*, \eta_1,\eta_2\in\C$,
	suppose that $V$ is a nonzero $\GG$-submodule of $\Omega(\lambda_1,\eta_1,\sigma_1,0)\otimes\Omega(\lambda_2,\eta_2,0,\sigma_2)$.
	Then $1\otimes 1\in V$.
\end{lemma}
\begin{proof}
	For any nonzero $v\in V$, we can write $v$ in the form 
	$$\sum_{i=0}^p\sum_{j=0}^q\sum_{k=0}^s\sum_{l=0}^t\alpha_{ijkl}X^iY^j\otimes S^kT^l,$$
	where $p,q,s,t\in\N, \alpha_{ijkl}\in\C$ and 
	$\sum_{i=0}^p\sum_{j=0}^q\sum_{k=0}^s\alpha_{ijkt}X^iY^j\otimes S^kT^t\ne 0$,
	which implies that 
	\begin{equation*}
	\sum_{i=0}^p\sum_{j=0}^q\sum_{k=0}^s\alpha_{ijkt}X^iY^j\otimes S^k\ne 0.
	\end{equation*}
	Denote $s'=\max\{ k~|~0\leq k\leq s, \sum_{i=0}^p\sum_{j=0}^q\alpha_{ijkt}X^iY^j\ne 0\}$.
	Thus 
	\begin{equation}\label{ij}
	\sum_{i=0}^p\sum_{j=0}^q\alpha_{ijs't}X^iY^j\ne 0,\quad  \alpha_{i'j's''t}=0,\quad 
	\forall 0\leq i'\leq p, 0\leq j'\leq q, s'<s''\leq s.
	\end{equation}
Let $m\in\Z$. We compute 
	\begin{align*}
	\lambda_2^{-m}\sigma_2^{-1}J_mv=\sum_{i=0}^p\sum_{j=0}^q\sum_{k=0}^s\sum_{l=0}^t\alpha_{ijkl}X^iY^j\otimes (S+1)^k(T-m)^l=\sum_{l=0}^tm^lv_l,
	\end{align*}
	where $v_0,v_1,\cdots, v_t\in \Omega(\lambda_1,\eta_1,\sigma_1,0)\otimes\Omega(\lambda_2,\eta_2,0,\sigma_2)$ are independent of $m$.
	Taking $m=1,2,\cdots, t+1$, then the coefficient matrix of $v_l$, where $l=0, 1, \cdots, t$, is a Vandermonde matrix.
	So $v_l\in V$ for any $0\leq l\leq t$. 
	In particular, 
	\begin{align*}
	v_t=\sum_{i=0}^p\sum_{j=0}^q\sum_{k=0}^{s'}(-1)^t\alpha_{ijkt}X^iY^j\otimes (S+1)^k
	=\sum_{i=0}^p\sum_{j=0}^q(-1)^t\alpha_{ijs't}X^iY^j\otimes S^{s'}+\text {lower-terms\ of\ } S,
	\end{align*}
	which is nonzero by Eq. ~\eqref{ij}.
	Hence we can assume that there exists nonzero 
	$$v'=\sum_{i=0}^p\sum_{j=0}^q\sum_{k=0}^s\beta_{ijk}X^iY^j\otimes S^k\in V,$$
	where $p,q,s\in\N, \beta_{ijk}\in\C$ and $\sum_{i=0}^p\sum_{j=0}^q\beta_{ijs}X^iY^j\otimes S^s\ne 0$. 
	If $s>0$, then for any $m\in\Z$, we have 
	\begin{align*}
	\lambda_2^{-m}\sigma_2^{-1}J_mv'-v'&=\sum_{i=0}^p\sum_{j=0}^q\sum_{k=0}^s\beta_{ijk}X^iY^j\otimes ((S+1)^k-S^k)\\
	&=\sum_{i=0}^p\sum_{j=0}^qs\beta_{ijs}X^iY^j\otimes S^{s-1}+\text {lower-terms\ of\ } S,
	\end{align*}
	which is nonzero. So we can get that there exists nonzero 
	$$v''=\sum_{i=0}^p\sum_{j=0}^q\gamma_{ij}X^iY^j\otimes 1\in V,$$
	where $p,q\in\N, \gamma_{ij}\in\C$. 
	Then using the action of $I_m$ on $v''$ where $m\in\Z$,  we can deduce that $1\otimes 1\in V$.
	This completes the proof of the lemma.
\end{proof}

\begin{proposition}\label{lambda12}
	Let $\lambda_1, \lambda_2, \sigma_1, \sigma_2\in\C^*, \eta_1,\eta_2\in\C$ with $\lambda_1\ne \lambda_2$.
	Then $$\<1\otimes 1\>=\Omega(\lambda_1,\eta_1,\sigma_1,0)\otimes\Omega(\lambda_2,\eta_2,0,\sigma_2),$$ 
	where $\<1\otimes 1\>$ denotes the $\GG$-submodule of $\Omega(\lambda_1,\eta_1,\sigma_1,0)\otimes\Omega(\lambda_2,\eta_2,0,\sigma_2)$ generated by $1\otimes 1$.
\end{proposition}
\begin{proof}
We only need to show that $$X^iY^j\otimes S^kT^l\in\<1\otimes 1\>,\quad {\rm \ for\ all\ } i,j,k,l\in\N.$$
First, it is clear that the above result holds when $i+j+k+l\leq 0$.
Now suppose that the conclusion holds when $i+j+k+l\leq n$, where $n\in\N$.
Let $i+j+k+l=n+1$. Then $$i\geq1, {\text\ or \ }j\geq 1,{\text\ or \ }k\geq 1,{\text\ or \ }l\geq 1.$$

Case 1: $i\geq 1$.	

Then $X^{i-1}Y^j\otimes S^kT^l\in\<1\otimes 1\>$.
For any $m\in\Z$, we have
\begin{align*}
H_m(X^{i-1}Y^j\otimes S^kT^l)=\lambda_1^mX^i(Y-m)^j\otimes S^kT^l+\lambda_2^mX^{i-1}Y^j\otimes S^{k+1}(T-m)^l\in\<1\otimes 1\>.
\end{align*}
By induction hypothesis we see that
$$\lambda_1^mX^iY^j\otimes S^kT^l+\lambda_2^mX^{i-1}Y^j\otimes S^{k+1}T^l\in\<1\otimes 1\>,\ \ {\rm \ for\ all\ \ } m\in\Z.$$
So $X^iY^j\otimes S^kT^l\in\<1\otimes 1\>$ since $\lambda_1\ne\lambda_2$.

Case 2: $j\geq 1$.	

Then $X^iY^{j-1}\otimes S^kT^l\in\<1\otimes 1\>$.
For any $m\in\Z$, we get 
\begin{align*}
&L_m(X^iY^{j-1}\otimes S^kT^l) \\
=&\lambda_1^mX^i(Y-m)^{j-1}(Y-mX+m\eta_1)\otimes S^kT^l+\lambda_2^mX^iY^{j-1}\otimes S^k(T-m)^l(T+mS+m\eta_2)\in\<1\otimes 1\>.
\end{align*}
By induction hypothesis we deduce 
$$\lambda_1^mX^iY^j\otimes S^kT^l+\lambda_2^mX^iY^{j-1}\otimes S^kT^{l+1}-m\lambda_1^mX^{i+1}Y^{j-1}\otimes S^kT^l+m\lambda_2^mX^iY^{j-1}\otimes S^{k+1}T^l\in\<1\otimes 1\>,$$
for any $m\in\Z$. Denote
\[
A=\begin{pmatrix}
1&1&0&0\\            
\lambda_1&\lambda_2&-\lambda_1&\lambda_2\\
\lambda_1^2&\lambda_2^2&-2\lambda_1^2&2\lambda_2^2\\
\lambda_1^3&\lambda_2^3&-3\lambda_1^3&3\lambda_2^3
\end{pmatrix}.
\]	
It is easy to see that the matrix $A$ is invertible. 
So $X^iY^j\otimes S^kT^l\in\<1\otimes 1\>$.
	
Case 3: $k\geq 1$ or $l\geq 1$.	

Then we also can get $X^iY^j\otimes S^kT^l\in\<1\otimes 1\>$ by similar discussions to Case 1 or Case 2.

In conclusion, we complete the proof of the proposition.	
\end{proof}

\begin{proposition}\label{lambda1=2}
	Suppose that $\lambda, \sigma_1, \sigma_2\in\C^*, \eta_1,\eta_2\in\C$.
	Let $V$ be a proper subspace of $\Omega(\lambda,\eta_1,\sigma_1,0)\otimes\Omega(\lambda,\eta_2,0,\sigma_2)$, spanned by $$\{\sum_{t=0}^j{\mathrm C}_j^tX^iY^{j-t}\otimes S^kT^t~|~i,j,k\in\N\},$$
where $\mathrm{C}_j^t=\frac{j!}{t!(j-t)!}$.
Then $V$ is a minimal proper submodule of $\Omega(\lambda,\eta_1,\sigma_1,0)\otimes\Omega(\lambda,\eta_2,0,\sigma_2)$.
\end{proposition}
\begin{proof}
It is sufficient to show that $V$ is a $\GG$-submodule of $\Omega(\lambda,\eta_1,\sigma_1,0)\otimes\Omega(\lambda,\eta_2,0,\sigma_2)$
and $V$ is minimal.
\begin{claim}\label{submodule}
$V$ is a $\GG$-submodule of $\Omega(\lambda,\eta_1,\sigma_1,0)\otimes\Omega(\lambda,\eta_2,0,\sigma_2)$.
\end{claim}
For any $i, j, k\in\N$, denote $v=\sum_{t=0}^j{\mathrm C}_j^tX^iY^{j-t}\otimes S^kT^t$. 
We only need to prove that 
$$I_mv,\  J_mv,\  H_mv,\  L_mv\in V, \ \  {\rm \ for\ any\ } m\in\Z.$$
Note that for any $\al,\al_1,\be,\be_1,\gamma,\gamma_1\in\N$, 
\begin{equation}
\left\{
\begin{aligned}\label{albega}
&{\mathrm C}_\al^\be{\mathrm C}_{\al-\be}^\gamma={\mathrm C}_\al^\gamma{\mathrm C}_{\al-\gamma}^\be,\,\,\ \text{if}\ \gamma\leq \al-\be,\ \  \be\leq \al,\\
&{\mathrm C}_\al^\be{\mathrm C}_\be^\gamma={\mathrm C}_\al^\gamma{\mathrm C}_{\al-\gamma}^{\be-\gamma}, \quad\ \text{if}\ \gamma\leq\be\leq\al.
\end{aligned}
\right.
\end{equation}
By direct computations and Eq.~\eqref{albega} we obtain that for any $m\in\Z$,
\begin{align}
\lambda^{-m}\sigma_1^{-1}I_mv&=\sum_{t=0}^j{\mathrm C}_j^t(X-1)^i(Y-m)^{j-t}\otimes S^kT^t  \notag \\
&=\sum_{t=0}^j(-m)^t{\mathrm C}_j^t(\sum_{l=0}^{j-t}{\mathrm C}_{j-t}^l(X-1)^iY^{j-t-l}\otimes S^kT^l),\label{Im}\\
\lambda^{-m}\sigma_2^{-1}J_mv&=\sum_{t=0}^j{\mathrm C}_j^tX^iY^{j-t}\otimes (S+1)^k(T-m)^t  \notag \\
&=\sum_{t=0}^j(-m)^t{\mathrm C}_j^t(\sum_{l=0}^{j-t}{\mathrm C}_{j-t}^lX^iY^{j-t-l}\otimes (S+1)^kT^l),   \label{Jm} \\
\lambda^{-m}H_mv&=\sum_{t=0}^j{\mathrm C}_j^tX^{i+1}(Y-m)^{j-t}\otimes S^kT^t+\sum_{t=0}^j{\mathrm C}_j^tX^iY^{j-t}\otimes S^{k+1}(T-m)^t.\label{Hm}
\end{align}
It is easy to see that $$\sum_{l=0}^{j-t}{\mathrm C}_{j-t}^l(X-1)^iY^{j-t-l}\otimes S^kT^l\in V, \quad
\sum_{l=0}^{j-t}{\mathrm C}_{j-t}^lX^iY^{j-t-l}\otimes (S+1)^kT^l\in V.$$
Then by observing Eqs.~\eqref{Im}, \eqref{Jm} and \eqref{Hm} we know that 
$$\lambda^{-m}\sigma_1^{-1}I_mv\in V,\ \ \lambda^{-m}\sigma_2^{-1}J_mv\in V,\ \ \lambda^{-m}H_mv\in V,\ \ 
{\rm \ for\ all\ } m\in\Z.$$
Now we compute 
\begin{align}
&\lambda^{-m}L_mv   \notag  \\
=&\sum_{t=0}^j{\mathrm C}_j^tX^i(Y-m)^{j-t}(Y-mX+m\eta_1)\otimes S^kT^t
+\sum_{t=0}^j{\mathrm C}_j^tX^iY^{j-t}\otimes S^k(T-m)^t(T+mS+m\eta_2)   \notag   \\
=&\sum_{t=0}^j{\mathrm C}_j^tX^i(Y-m)^{j-t}Y\otimes S^kT^t
-m\sum_{t=0}^j{\mathrm C}_j^tX^{i+1}(Y-m)^{j-t}\otimes S^kT^t  \notag \\
&+m\eta_1\sum_{t=0}^j{\mathrm C}_j^tX^i(Y-m)^{j-t}\otimes S^kT^t
+\sum_{t=0}^j{\mathrm C}_j^tX^iY^{j-t}\otimes S^k(T-m)^tT    \notag  \\
&+m\sum_{t=0}^j{\mathrm C}_j^tX^iY^{j-t}\otimes S^{k+1}(T-m)^t
+m\eta_2\sum_{t=0}^j{\mathrm C}_j^tX^iY^{j-t}\otimes S^k(T-m)^t, \ \ \ {\rm \ for\ all\ } m\in\Z. \label{Lm}  
\end{align}
By Eqs.~\eqref{Im} and \eqref{Jm} we see that 
\begin{align}
&-m\sum_{t=0}^j{\mathrm C}_j^tX^{i+1}(Y-m)^{j-t}\otimes S^kT^t\in V,\ \  m\eta_1\sum_{t=0}^j{\mathrm C}_j^tX^i(Y-m)^{j-t}\otimes S^kT^t\in V,
\label{Lm23}  \\
&m\sum_{t=0}^j{\mathrm C}_j^tX^iY^{j-t}\otimes S^{k+1}(T-m)^t\in V,\ \  m\eta_2\sum_{t=0}^j{\mathrm C}_j^tX^iY^{j-t}\otimes S^k(T-m)^t\in V,
\label{Lm56}\\
&\sum_{t=0}^j{\mathrm C}_j^tX^i(Y-m)^{j-t}Y\otimes S^kT^t=
\sum_{t=0}^j(-m)^t{\mathrm C}_j^t(\sum_{l=0}^{j-t}{\mathrm C}_{j-t}^lX^iY^{j+1-t-l}\otimes S^kT^l),\\
&\sum_{t=0}^j{\mathrm C}_j^tX^iY^{j-t}\otimes S^k(T-m)^tT=
\sum_{t=0}^j(-m)^t{\mathrm C}_j^t(\sum_{l=0}^{j-t}{\mathrm C}_{j-t}^lX^iY^{j-t-l}\otimes S^kT^{l+1}).
\end{align}
Then we have
\begin{align}
&\sum_{t=0}^j{\mathrm C}_j^tX^i(Y-m)^{j-t}Y\otimes S^kT^t+\sum_{t=0}^j{\mathrm C}_j^tX^iY^{j-t}\otimes S^k(T-m)^tT \notag \\
=&\sum_{t=0}^j(-m)^t{\mathrm C}_j^t(\sum_{l=0}^{j-t}{\mathrm C}_{j-t}^lX^iY^{j+1-t-l}\otimes S^kT^l)
+\sum_{t=0}^j(-m)^t{\mathrm C}_j^t(\sum_{l=1}^{j+1-t}{\mathrm C}_{j-t}^{l-1}X^iY^{j+1-t-l}\otimes S^kT^l)  \notag \\
=&\sum_{t=0}^j(-m)^t{\mathrm C}_j^t(\sum_{l=1}^{j-t}({\mathrm C}_{j-t}^l+{\mathrm C}_{j-t}^{l-1})X^iY^{j+1-t-l}\otimes S^kT^l)
+\sum_{t=0}^j(-m)^t{\mathrm C}_j^t(X^iY^{j+1-t}\otimes S^k+X^i\otimes S^kT^{j+1-t})   \notag   \\
=&\sum_{t=0}^j(-m)^t{\mathrm C}_j^t(\sum_{l=0}^{j+1-t}{\mathrm C}_{j+1-t}^lX^iY^{j+1-t-l}\otimes S^kT^l)\in V.  \label{Lm14}
\end{align}
Hence $L_mv\in V$ by Eqs.~\eqref{Lm}-\eqref{Lm56} and \eqref{Lm14}. Claim \ref{submodule} is proved.

Next, we prove that $V$ is minimal. Note that any nonzero submodule of $\Omega(\lambda_1,\eta_1,\sigma_1,0)\otimes\Omega(\lambda_2,\eta_2,0,\sigma_2)$ contains $1\otimes 1$.
So we only need to show that $V$ can be generated by $1\otimes 1$.
\begin{claim}\label{XS}
	$X^i\otimes S^k\in\<1\otimes 1\>$ for any $i,k\in\N$.
\end{claim}
Obviously, $1\otimes 1\in\<1\otimes 1\>$. 
Now fix $i,k\in\N$, suppose that $X^i\otimes S^k\in\<1\otimes 1\>$.
It is sufficient to prove that $X^{i+1}\otimes S^k, X^i\otimes S^{k+1}\in\<1\otimes 1\>$.
For any $m\in\Z$, we compute
\begin{align*}
\lambda^{-m}L_m(X^i\otimes S^k)&=X^i(Y-mX+m\eta_1)\otimes S^k+X^i\otimes S^k(T+mS+m\eta_2)\\
&=X^iY\otimes S^k+X^i\otimes S^kT-m(X^{i+1}\otimes S^k-X^i\otimes S^{k+1}-\eta_1X^i\otimes S^k-\eta_2 X^i\otimes S^k),
\end{align*}
which implies that 
$X^{i+1}\otimes S^k-X^i\otimes S^{k+1}-\eta_1X^i\otimes S^k-\eta_2 X^i\otimes S^k\in\<1\otimes 1\>$.
So 
\begin{align}\label{XS-1}
X^{i+1}\otimes S^k-X^i\otimes S^{k+1}\in\<1\otimes 1\>,
\end{align}
 since $X^i\otimes S^k\in\<1\otimes 1\>$.
Moreover, 
\begin{align}\label{XS-2}
H_0(X^i\otimes S^k)=X^{i+1}\otimes S^k+X^i\otimes S^{k+1}\in\<1\otimes 1\>.
\end{align}
By Eqs.~\eqref{XS-1} and \eqref{XS-2} we get $X^{i+1}\otimes S^k, X^i\otimes S^{k+1}\in\<1\otimes 1\>$.
 This proves the Claim \ref{XS}.
\begin{claim}\label{XYS}
For any $j\in\N$,	$\{\sum_{t=0}^j{\mathrm C}_j^tX^iY^{j-t}\otimes S^kT^t~|~i,k\in\N\}\subseteq\<1\otimes 1\>$. 
\end{claim}
By Claim \ref{XS} the result holds when $j=0$.
Now suppose that the result holds when $j=n$, where $n\in\N$. Then we obtain
\begin{align*}
L_0(\sum_{t=0}^n{\mathrm C}_n^tX^iY^{n-t}\otimes S^kT^t)
&=\sum_{t=0}^n{\mathrm C}_n^tX^iY^{n+1-t}\otimes S^kT^t+\sum_{t=0}^n{\mathrm C}_n^tX^iY^{n-t}\otimes S^kT^{t+1}\\
&=\sum_{t=0}^n{\mathrm C}_n^tX^iY^{n+1-t}\otimes S^kT^t+\sum_{t=1}^{n+1}{\mathrm C}_n^{t-1}X^iY^{n+1-t}\otimes S^kT^t\\
&=X^iY^{n+1}\otimes S^k+\sum_{t=1}^n{\mathrm C}_{n+1}^tX^iY^{n+1-t}\otimes S^kT^t+X^i\otimes S^kT^{n+1}\\
&=\sum_{t=0}^{n+1}{\mathrm C}_{n+1}^tX^iY^{n+1-t}\otimes S^kT^t, \ \ \ {\rm \ for\ all \ }\ i,k\in\N.
\end{align*}
Thus Claim \ref{XYS} is proved.

From Claim \ref{XYS} we see that $V\subseteq\<1\otimes 1\>$. 
Therefore $V$ can be generated by $1\otimes 1$.
This completes the proof of the proposition.
\end{proof}

By Lemma \ref{nonzeroV} and Propositions \ref{lambda12}, \ref{lambda1=2} we can directly obtain the following theorem.
\begin{theorem}\label{irr-12}
Let $\lambda_1, \lambda_2, \sigma_1, \sigma_2\in\C^*$ and $\eta_1,\eta_2\in\C$.
Then $\Omega(\lambda_1,\eta_1,\sigma_1,0)\otimes\Omega(\lambda_2,\eta_2,0,\sigma_2)$ is an irreducible $\GG$-module if and only if 
$\lambda_1\ne \lambda_2$.
\end{theorem}

\subsection{Isomorphism classes}
In this subsection, we determine the necessary and sufficient conditions 
for two of irreducible tensor product modules obtained in subsection \ref{sub3-1} to be isomorphic.

Note that for any $v\in\Omega(\lambda_1,\eta_1,\sigma_1,0)\otimes\Omega(\lambda_2,\eta_2,0,\sigma_2)$, we can write $v$ in the form 
$$\sum_{i=0}^p\sum_{j=0}^q\sum_{k=0}^s\sum_{l=0}^t\alpha_{ijkl}X^iY^j\otimes S^kT^l,$$
where $p,q,s,t\in\N,\al_{ijkl}\in\C$.
Now let $\Supp(v)$ denote the set of all $(i,j,k,l)$ with $\alpha_{ijkl}\ne 0$
and $\deg(v)$ be the maximal element of $\Supp(v)$ with respect to the total order $\succ$ on $\N^4$,
called the {\bf degree} of $v$.

\begin{theorem}\label{iso-12}
	Let $\lambda_1, \lambda_2,\lambda_1', \lambda_2', \sigma_1, \sigma_2, \sigma_1', \sigma_2'\in\C^*$ and $\eta_1,\eta_2,\eta_1',\eta_2'\in\C$ with $\lambda_1\ne\lambda_2, \lambda_1'\ne\lambda_2'$.
	Then the irreducible $\GG$-modules $$\Omega(\lambda_1,\eta_1,\sigma_1,0)\otimes\Omega(\lambda_2,\eta_2,0,\sigma_2){\rm \ \  and\ \ } \Omega(\lambda_1',\eta_1',\sigma_1',0)\otimes\Omega(\lambda_2',\eta_2',0,\sigma_2')$$
	are isomorphic if and only if 
	$$(\lambda_1,\eta_1,\sigma_1)=(\lambda_1',\eta_1',\sigma_1'){\rm\ \  and \ \ } (\lambda_2,\eta_2,\sigma_2)=(\lambda_2',\eta_2',\sigma_2').$$
\end{theorem}
\begin{proof}
The ``if part" is trivial. We only need to show the ``only if part".	
Let $$\varphi:\Omega(\lambda_1,\eta_1,\sigma_1,0)\otimes\Omega(\lambda_2,\eta_2,0,\sigma_2)\rightarrow
\Omega(\lambda_1',\eta_1',\sigma_1',0)\otimes\Omega(\lambda_2',\eta_2',0,\sigma_2')$$ be a $\GG$-module isomorphism.
Then we can write $$\varphi(1\otimes 1)=\sum_{i=0}^p\sum_{j=0}^q\sum_{k=0}^s\sum_{l=0}^t\be_{ijkl}X^iY^j\otimes S^kT^l,$$
where $p,q,s,t\in\N,\be_{ijkl}\in\C$ and $\be_{ijkl}$ are not all zero.
Denote $\deg(\varphi(1\otimes 1))=(p',q',s',t')$, where $$0\leq p'\leq p,\  0\leq q'\leq q,\  0\leq s'\leq s {\rm \ \ and\ \ } 0\leq t'\leq t.$$ So $\be_{p'q's't'}\ne 0$.
Note that
\begin{align}
0= &I_m\varphi(1\otimes 1)-\varphi(I_m(1\otimes 1))  \notag\\
=&\sum_{i=0}^p\sum_{j=0}^q\sum_{k=0}^s\sum_{l=0}^t([(\lambda_1')^m\sigma_1'\be_{ijkl}(X-1)^i(Y-m)^j-\lambda_1^m\sigma_1\be_{ijkl}X^iY^j]\otimes S^kT^l)  \notag      \\
=&[(\lambda_1')^m\sigma_1'-\lambda_1^m\sigma_1]\be_{p'q's't'}X^{p'}Y^{q'}\otimes S^{s'}T^{t'} + v', {\rm \ \ \ for \ \ all\ \ } m\in\Z,  \label{iso-Im}
\end{align}
where $v'\in\Omega(\lambda_1',\eta_1',\sigma_1',0)\otimes\Omega(\lambda_2',\eta_2',0,\sigma_2')$ and 
$$ v'=0, {\rm \ or\ } (p',q',s',t')\succ\deg(v').$$
Thus $$(\lambda_1')^m\sigma_1'-\lambda_1^m\sigma_1=0, \ \ {\rm \ for\ any\ } m\in\Z,$$ which imply 
\begin{align}\label{11'11'}
\lambda_1=\lambda_1',\ \ \sigma_1=\sigma_1'.
\end{align}
\begin{claim}\label{p'q'}
	$p'=0, q'=0$.
\end{claim}
Assuming $p'>0$. Then substituting Eq.~\eqref{11'11'} into Eq.~\eqref{iso-Im}, we obtain that for any $m\in\Z$,
\begin{align*}
0= &I_m\varphi(1\otimes 1)-\varphi(I_m(1\otimes 1))  \notag\\
=&-p'\lambda_1^m\sigma_1\be_{p'q's't'}X^{p'-1}Y^{q'}\otimes S^{s'}T^{t'} + v'',     
\end{align*}
where $v''\in\Omega(\lambda_1',\eta_1',\sigma_1',0)\otimes\Omega(\lambda_2',\eta_2',0,\sigma_2')$ and 
$$v''=0, {\rm\ or\ }(p'-1,q',s',t')\succ\deg(v''),$$
which yields that $-p'\lambda_1^m\sigma_1\be_{p'q's't'}=0$. This is a contradiction. So $p'=0$.

Assuming $q'>0$. Then substituting $p'=0$ and Eq.~\eqref{11'11'} into Eq.~\eqref{iso-Im}, we deduce that for any $m\in\Z$,
\begin{align*}
0= &I_m\varphi(1\otimes 1)-\varphi(I_m(1\otimes 1))  \notag\\
=&-\lambda_1^m\sigma_1(mq'\be_{0q's't'}+\be_{1(q'-1)s't'})Y^{q'-1}\otimes S^{s'}T^{t'} + v''',     
\end{align*}
where $v'''\in\Omega(\lambda_1',\eta_1',\sigma_1',0)\otimes\Omega(\lambda_2',\eta_2',0,\sigma_2')$ and 
$$v'''=0, {\rm\ or\ }(0,q'-1,s',t')\succ\deg(v''').$$
Thus $mq'\be_{0q's't'}+\be_{1(q'-1)s't'}=0$ for any $m\in\Z$, 
which is impossible since $\be_{0q's't'}\ne 0$. So $q'=0$.
Claim \ref{p'q'} is proved. 

Similarly, using that $$0=J_m\varphi(1\otimes 1)-\varphi(J_m(1\otimes 1)) , \ {\rm \ for\ any\ } m\in\Z,$$
we can deduce $$\lambda_2=\lambda_2',\ \sigma_2=\sigma_2',\ s'=0\ {\rm \ and\ }\ t'=0.$$

Now we get that 
\begin{align}\label{1-1}
\varphi(1\otimes 1)=\al(1\otimes 1),
\end{align}
 where $\al\in\C^*$. Then for any $m\in\Z$, we compute
\begin{align*}
\lambda_1^m\varphi(X\otimes 1)+\lambda_2^m\varphi(1\otimes S)&=\varphi(H_m(1\otimes 1))\\
&=H_m(\varphi(1\otimes 1))=\al\lambda_1^m(X\otimes 1)+\al\lambda_2^m(1\otimes S),
\end{align*}
which imply that
\begin{align}\label{X-11-S}
\varphi(X\otimes 1)=\al(X\otimes 1) {\rm \ \ and\ \ } \varphi(1\otimes S)=\al(1\otimes S),
\end{align}
since $\lambda_1\ne\lambda_2$.
Finally, using Eqs.~\eqref{1-1} and \eqref{X-11-S} we see that
\begin{align*}
0=&\varphi(L_m(1\otimes 1))-L_m\varphi(1\otimes 1)\\
=&\lambda_1^m(\varphi(Y\otimes 1)-\al Y\otimes 1)+\lambda_2^m(\varphi(1\otimes T)-\al\otimes T)  \\
&+m\lambda_1^m(\eta_1-\eta_1')\al(1\otimes 1)
+m\lambda_2^m(\eta_2-\eta_2')\al(1\otimes 1),\ \ \ {\rm for\ all \ } m\in\Z.
\end{align*}
Taking $m=0,1,2,3$, we obtain that
	\[
\begin{pmatrix}
1&1&0&0\\            
\lambda_1&\lambda_2&\lambda_1&\lambda_2\\
\lambda_1^2&\lambda_2^2&2\lambda_1^2&2\lambda_2^2\\
\lambda_1^3&\lambda_2^3&3\lambda_1^3&3\lambda_2^3
\end{pmatrix}
\begin{pmatrix}
\varphi(Y\otimes 1)-\al Y\otimes 1\\            
\varphi(1\otimes T)-\al\otimes T\\
(\eta_1-\eta_1')\al(1\otimes 1)\\
(\eta_2-\eta_2')\al(1\otimes 1)
\end{pmatrix}=
\begin{pmatrix}
0\\            
0\\
0\\
0
\end{pmatrix}.
\]
It is easy to see that the matrix
\[
\begin{pmatrix}
1&1&0&0\\            
\lambda_1&\lambda_2&\lambda_1&\lambda_2\\
\lambda_1^2&\lambda_2^2&2\lambda_1^2&2\lambda_2^2\\
\lambda_1^3&\lambda_2^3&3\lambda_1^3&3\lambda_2^3
\end{pmatrix}
\]
is invertible since $\lambda_1\ne\lambda_2$. So
$$\varphi(Y\otimes 1)=\al Y\otimes 1,\ \varphi(1\otimes T)=\al\otimes T,\ \eta_1=\eta_1'\ {\rm \ and \ }\ \eta_2=\eta_2'.$$

In conclusion, we complete the proof of the theorem.
\end{proof}

From the above theorem, we have the following corollary, which is one of the main results of \cite{CGZ}.
\begin{corollary}
	 Let $\lambda,\lambda_1,\sigma,\sigma_1\in\C^*$ and $\eta,\eta_1\in\C$.  Then the following statements hold.
	\begin{enumerate}[$(1)$]
		\item $\Omega(\lambda,\eta,\sigma,0)$ and $\Omega(\lambda_1,\eta_1,\sigma_1,0)$ are isomorphic as $\GG$-modules if and only if $\lambda=\lambda_1,\eta=\eta_1,\sigma=\sigma_1$.
		\item $\Omega(\lambda,\eta,0,\sigma)$ and $\Omega(\lambda_1,\eta_1,0,\sigma_1)$ are isomorphic as $\GG$-modules if and only if $\lambda=\lambda_1,\eta=\eta_1,\sigma=\sigma_1$.
	\end{enumerate}
\end{corollary}

\section{$\Omega(\lambda_1,\eta_1,\sigma_1,0)\otimes\Omega(\lambda_2,\eta_2,\sigma_2,0)$ and $\Omega(\lambda_1,\eta_1,0,\sigma_1)\otimes\Omega(\lambda_2,\eta_2,0,\sigma_2)$}\label{Sec4}
In this section, for convenience, we denote
$$\Omega(\lambda_1,\eta_1,\sigma_1,0)=\C[X,Y] {\rm \ \ \ and\ \ \ } \Omega(\lambda_2,\eta_2,\sigma_2,0)=\C[X_1,Y_1],$$
which are all $\GG$-modules, where $\lambda_1,\lambda_2,\sigma_1, \sigma_2\in\C^*$ and $\eta_1,\eta_2\in\C$.
We investigate the tensor product module $\Omega(\lambda_1,\eta_1,\sigma_1,0)\otimes\Omega(\lambda_2,\eta_2,\sigma_2,0)$.
The isomorphism classes and conditions for these tensor product modules to be irreducible are determined.
Furthermore, the results on tensor product module $\Omega(\lambda_1,\eta_1,0,\sigma_1)\otimes\Omega(\lambda_2,\eta_2,0,\sigma_2)$ follow the same way.

\subsection{Irreducibility}\label{sub4-1}
In this subsection, we study the irreducibility of tensor product modules $\Omega(\lambda_1,\eta_1,\sigma_1,0)\otimes\Omega(\lambda_2,\eta_2,\sigma_2,0)$
and $\Omega(\lambda_1,\eta_1,0,\sigma_1)\otimes\Omega(\lambda_2,\eta_2,0,\sigma_2)$.

For any nonzero element $v\in\Omega(\lambda_1,\eta_1,\sigma_1,0)\otimes\Omega(\lambda_2,\eta_2,\sigma_2,0)$, we can write $v$ in the form 
\begin{align}\label{beta}
\sum_{i=0}^p\sum_{j=0}^q\sum_{k=0}^s\sum_{l=0}^t\be_{ijkl}X^iY^j\otimes X_1^kY_1^l,
\end{align}
where $p,q,s,t\in\N,\be_{ijkl}\in\C$.
Let $\Supp(v)$ denote the set of all $(i,j,k,l)$ with $\be_{ijkl}\ne 0$
and $\deg(v)$ be the maximal element of $\Supp(v)$ with respect to the total order $\succ$ on $\N^4$,
called the {\bf degree} of $v$.

\begin{lemma}\label{lower-degree}
Let $\lambda_1, \lambda_2, \sigma_1, \sigma_2\in\C^*$ and $\eta_1,\eta_2\in\C$ with $\lambda_1\ne\lambda_2$.
Suppose that $v\in\Omega(\lambda_1,\eta_1,\sigma_1,0)\otimes\Omega(\lambda_2,\eta_2,\sigma_2,0)$ is nonzero.
	Denote $\deg(v)=(p',q',s',t')$, where $p',q',s', t'\in\N$.
	The following statements hold.
		\begin{enumerate}[$(1)$]
		\item If $p'>0$, then $$\deg(I_0v-\sigma_1v-\sigma_2v)=(p'-1,q',s',t').$$
		\item If $p'=0, q'>0$, then there exists $m\in\Z$ such that $$\deg(I_mv-\lambda_1^m\sigma_1v-\lambda_2^m\sigma_2v)=(0,q'-1,s',t').$$
		\item If $p'=q'=0, s'>0$, then there exists $m\in\Z$ such that $$\deg(I_mv-\lambda_1^m\sigma_1v-\lambda_2^m\sigma_2v)=(0,0,s'-1,t').$$
		\item If $p'=q'=s'=0, t'>0$, then there exists $m\in\Z$ such that $$\deg(I_mv-\lambda_1^m\sigma_1v-\lambda_2^m\sigma_2v)=(0,0,0,t'-1).$$
	\end{enumerate}
\end{lemma}
\begin{proof}
	We first write $v$ in the form \eqref{beta}. 
Then it is clear that $p\geq p', q\geq q', s\geq s', t\geq t'$ and $\be_{p'q's't'}\ne 0$ since $\deg(v)=(p',q',s',t')$.
	For any $m\in\Z$, we compute
	\begin{align}
	&I_mv-\lambda_1^m\sigma_1v-\lambda_2^m\sigma_2v \notag \\
	=&\sum_{i=0}^p\sum_{j=0}^q\sum_{k=0}^s\sum_{l=0}^t\lambda_1^m\sigma_1\be_{ijkl}((X-1)^i(Y-m)^j-X^iY^j)\otimes X_1^kY_1^l \notag \\
	&+\sum_{i=0}^p\sum_{j=0}^q\sum_{k=0}^s\sum_{l=0}^t\lambda_2^m\sigma_2\be_{ijkl}X^iY^j\otimes((X_1-1)^k(Y_1-m)^l-X_1^kY_1^l). \label{Im-v-v}
	\end{align}

$(1)$ If $p'>0$, then it is easy to see that
$$I_0v-\sigma_1v-\sigma_2v=-p'\sigma_1\be_{p'q's't'}X^{p'-1}Y^{q'}\otimes X_1^{s'}Y_1^{t'}+{\rm \ lower-terms\ }.$$
So $\deg(I_0v-\sigma_1v-\sigma_2v)=(p'-1,q',s',t')$.

$(2)$ If $p'=0, q'>0$, then there exists $m\in\Z$ such that $$mq'\be_{0q's't'}+\be_{1(q'-1)s't'}\ne 0,$$ since $\be_{0q's't'}\ne 0$. So there exists $m\in\Z$ such that
$$I_mv-\lambda_1^m\sigma_1v-\lambda_2^m\sigma_2v 
=(-mq'\be_{0q's't'}-\be_{1(q'-1)s't'})\lambda_1^m\sigma_1Y^{q'-1}\otimes X_1^{s'}Y_1^{t'}+{\rm \ lower-terms\ },$$
 which yields that there exists $m\in\Z$ such that $$\deg(I_mv-\lambda_1^m\sigma_1v-\lambda_2^m\sigma_2v)=(0,q'-1,s',t').$$

$(3)$ Suppose that $p'=q'=0, s'>0$. Then it is easy to see that the following matrix
\[
\begin{pmatrix}
-1&-1&0\\            
-\lambda_2&-\lambda_1&-\lambda_1\\
-\lambda_2^2&-\lambda_1^2&-2\lambda_1^2
\end{pmatrix}
\]
is invertible since $\lambda_1\ne\lambda_2$. 
Thus there exists $0\leq m\leq 2$ such that
$$-\lambda_2^m\sigma_2s'\be_{00s't'}-\lambda_1^m\sigma_1\be_{10(s'-1)t'}-m\lambda_1^m\sigma_1\be_{01(s'-1)t'}\ne 0,$$
since $\sigma_2s'\be_{00s't'}\ne 0$.
So there exists $0\leq m\leq 2$ such that
\begin{align*}
&I_mv-\lambda_1^m\sigma_1v-\lambda_2^m\sigma_2v   \\
=&(-\lambda_2^m\sigma_2s'\be_{00s't'}-\lambda_1^m\sigma_1\be_{10(s'-1)t'}-m\lambda_1^m\sigma_1\be_{01(s'-1)t'})\otimes X_1^{s'-1}Y_1^{t'}+{\rm \ lower-terms\ },
\end{align*}
which implies that there exists $0\leq m\leq 2$ such that $$\deg(I_mv-\lambda_1^m\sigma_1v-\lambda_2^m\sigma_2v)=(0,0,s'-1,t').$$

$(4)$ Suppose that $p'=q'=s'=0, t'>0$. Then the following matrix
\[
\begin{pmatrix}
0&-1&0&-1\\            
-\lambda_2&-\lambda_2&-\lambda_1&-\lambda_1\\
-2\lambda_2^2&-\lambda_2^2&-2\lambda_1^2&-\lambda_1^2\\
-3\lambda_2^3&-\lambda_2^3&-3\lambda_1^3&-\lambda_1^3
\end{pmatrix}
\]
is invertible since $\lambda_1\ne\lambda_2$. 
Thus there exists $0\leq m\leq 3$ such that
$$-m\lambda_2^m\sigma_2t'\be_{000t'}-\lambda_2^m\sigma_2\be_{001(t'-1)}
-m\lambda_1^m\sigma_1\be_{010(t'-1)}-\lambda_1^m\sigma_1\be_{100(t'-1)}\ne 0,$$
since $\sigma_2t'\be_{000t'}\ne 0$.
So there exists $0\leq m\leq 3$ such that
\begin{align*}
&I_mv-\lambda_1^m\sigma_1v-\lambda_2^m\sigma_2v   \\
=&(-m\lambda_2^m\sigma_2t'\be_{000t'}-\lambda_2^m\sigma_2\be_{001(t'-1)}
-m\lambda_1^m\sigma_1\be_{010(t'-1)}-\lambda_1^m\sigma_1\be_{100(t'-1)})\otimes Y_1^{t'-1}+{\rm \ lower-terms\ },
\end{align*}
which shows that there exists $0\leq m\leq 3$ such that $$\deg(I_mv-\lambda_1^m\sigma_1v-\lambda_2^m\sigma_2v)=(0,0,0,t'-1).$$
This completes the proof of the lemma.
\end{proof}

From the above lemma, the following proposition is clear.
\begin{proposition}\label{nonzeroV-1}
	Let $\lambda_1, \lambda_2, \sigma_1, \sigma_2\in\C^*$ and $\eta_1,\eta_2\in\C$ with $\lambda_1\ne\lambda_2$.
	Suppose that $V$ is a nonzero $\GG$-submodule of $\Omega(\lambda_1,\eta_1,\sigma_1,0)\otimes\Omega(\lambda_2,\eta_2,\sigma_2,0)$.
	Then $1\otimes 1\in V$.
\end{proposition}

Recall that we only use the actions of $H_m$ and $L_m$ in the proof of Proposition \ref{lambda12}, where $m\in\Z$.
So by the similar discussions we can get the following proposition.
\begin{proposition}\label{lambda12'}
	Let $\lambda_1, \lambda_2, \sigma_1, \sigma_2\in\C^*$ and $\eta_1,\eta_2\in\C$ with $\lambda_1\ne\lambda_2$.
	 Then $$\<1\otimes 1\>=\Omega(\lambda_1,\eta_1,\sigma_1,0)\otimes\Omega(\lambda_2,\eta_2,\sigma_2,0),$$ 
	where $\<1\otimes 1\>$ denotes the $\GG$-submodule of $\Omega(\lambda_1,\eta_1,\sigma_1,0)\otimes\Omega(\lambda_2,\eta_2,\sigma_2,0)$ generated by $1\otimes 1$.
\end{proposition}

\begin{proposition}\label{lambda1=2'}
Let $\lambda, \sigma_1, \sigma_2\in\C^*, \eta_1,\eta_2\in\C$.
 Denote $$W=\mathrm{span}\{\sum_{t=0}^j{\mathrm C}_j^tX^iY^{j-t}\otimes X_1^kY_1^t~|~i,j,k\in\N\},$$
which is a nonzero proper subspace of $\Omega(\lambda,\eta_1,\sigma_1,0)\otimes\Omega(\lambda,\eta_2,\sigma_2,0)$.
Then $W$ is a nonzero proper submodule of $\Omega(\lambda,\eta_1,\sigma_1,0)\otimes\Omega(\lambda,\eta_2,\sigma_2,0)$.
\end{proposition}
\begin{proof}
	 It is sufficient to show that for any $i, j, k\in\N$,  
	$$I_mw, J_mw, H_mw, L_mw\in W, \quad {\rm \ for\ all\ } m\in\Z,$$
where $w=\sum_{t=0}^j{\mathrm C}_j^tX^iY^{j-t}\otimes X_1^kY_1^t$.
	 First,  for any $m\in\Z$, $J_mw\in W$ is obvious since $J_mw=0$. 
	 Moreover, we have
	\begin{align*}
	\lambda^{-m}I_mw 
	&=\sigma_1\sum_{t=0}^j{\mathrm C}_j^t(X-1)^i(Y-m)^{j-t}\otimes X_1^kY_1^t 
	+\sigma_2\sum_{t=0}^j{\mathrm C}_j^tX^iY^{j-t}\otimes (X_1-1)^k(Y_1-m)^t, 
	\end{align*}
	for any $m\in\Z$.
	Then by Eqs.~\eqref{Im} and \eqref{Jm} we deduce that 
	\begin{align*}
	I_mw\in W,\ \  {\rm \ for\ all\ } m\in\Z.
	\end{align*}
	 Similarly, we can get that 
	\begin{align*}
 H_mw,\ L_mw\in W,\ \  {\rm \ for\ all\ } m\in\Z.
	\end{align*}
Therefore $W$ is a $\GG$-submodule of $\Omega(\lambda,\eta_1,\sigma_1,0)\otimes\Omega(\lambda,\eta_2,\sigma_2,0)$.
\end{proof}	

Now we combine Propositions \ref{nonzeroV-1}, \ref{lambda12'} and \ref{lambda1=2'} into the following main theorem.	
\begin{theorem}\label{irr-11}
	Let $\lambda_1, \lambda_2, \sigma_1, \sigma_2\in\C^*$ and $\eta_1,\eta_2\in\C$.
	Then $\Omega(\lambda_1,\eta_1,\sigma_1,0)\otimes\Omega(\lambda_2,\eta_2,\sigma_2,0)$ is an irreducible $\GG$-module if and only if $\lambda_1\ne \lambda_2$.
\end{theorem}	

By definitions of $\GG$-modules $\Omega(\lambda_1,\eta_1,\sigma_1,0), \Omega(\lambda_1,\eta_1,0,\sigma_1)$ and Lie brackets of $\GG$ 
the following theorem can be deduced via similar proceedings to Theorem \ref{irr-11}.
\begin{theorem}\label{irr-22}
	Let $\lambda_1, \lambda_2, \sigma_1, \sigma_2\in\C^*$ and $\eta_1,\eta_2\in\C$.
	Then $\Omega(\lambda_1,\eta_1,0,\sigma_1)\otimes\Omega(\lambda_2,\eta_2,0,\sigma_2)$ is an irreducible $\GG$-module if and only if $\lambda_1\ne \lambda_2$.
\end{theorem}

\subsection{Isomorphism classes}
In this subsection, we get the isomorphism classes of irreducible tensor product modules $\Omega(\lambda_1,\eta_1,\sigma_1,0)\otimes\Omega(\lambda_2,\eta_2,\sigma_2,0)$ and $\Omega(\lambda_1,\eta_1,0,\sigma_1)\otimes\Omega(\lambda_2,\eta_2,0,\sigma_2)$ over $\GG$ respectively,
where $\lambda_1, \lambda_2, \sigma_1, \sigma_2\in\C^*$ and $\eta_1,\eta_2\in\C$ with $\lambda_1\ne\lambda_2$.

Suppose that $V$ and $W$ are $\GG$-modules. Then we have the tensor product $\GG$-modules $V\otimes W$ and $W\otimes V$. It is well-known that the map
$$\xi: V\otimes W\rightarrow W\otimes V,\quad v\otimes w\mapsto w\otimes v,\ \ \ \ 
{\rm for\ \ any \ \ }v\in V,\ w\in W,$$
is a $\GG$-module isomorphism. 

\begin{theorem}\label{iso-11}
	Let $\lambda_1, \lambda_2,\lambda_1', \lambda_2', \sigma_1, \sigma_2, \sigma_1', \sigma_2'\in\C^*$ and $\eta_1,\eta_2,\eta_1',\eta_2'\in\C$ with $\lambda_1\ne\lambda_2, \lambda_1'\ne\lambda_2'$.
	Then the irreducible $\GG$-modules $$\Omega(\lambda_1,\eta_1,\sigma_1,0)\otimes\Omega(\lambda_2,\eta_2,\sigma_2,0){\rm \ \  and\ \ } \Omega(\lambda_1',\eta_1',\sigma_1',0)\otimes\Omega(\lambda_2',\eta_2',\sigma_2',0)$$
	are isomorphic if and only if 
	$$(\lambda_1,\eta_1,\sigma_1)=(\lambda_1',\eta_1',\sigma_1'), \ \ (\lambda_2,\eta_2,\sigma_2)=(\lambda_2',\eta_2',\sigma_2'),$$
	or 
	$$(\lambda_1,\eta_1,\sigma_1)=(\lambda_2',\eta_2',\sigma_2'), \ \ (\lambda_2,\eta_2,\sigma_2)=(\lambda_1',\eta_1',\sigma_1').$$
\end{theorem}
\begin{proof}
	$(\Leftarrow)$ This direction is trivial.
	
	$(\Rightarrow)$
	Suppose that $$\phi:\Omega(\lambda_1,\eta_1,\sigma_1,0)\otimes\Omega(\lambda_2,\eta_2,\sigma_2,0)\rightarrow
	\Omega(\lambda_1',\eta_1',\sigma_1',0)\otimes\Omega(\lambda_2',\eta_2',\sigma_2',0)$$ is a $\GG$-module isomorphism.
	Then we can denote $$\phi(1\otimes 1)=\sum_{i=0}^a\sum_{j=0}^b\sum_{k=0}^c\sum_{l=0}^d\gamma_{ijkl}X^iY^j\otimes X_1^kY_1^l,$$
	where $a,b,c,d\in\N,\gamma_{ijkl}\in\C$ and $\gamma_{ijkl}$ are not all zero.
	Denote $\deg(\phi(1\otimes 1))=(a',b',c',d')$, where $$0\leq a'\leq a,\  0\leq b'\leq b,\  0\leq c'\leq c {\rm \ \ and\ \ } 0\leq d'\leq d.$$ So $\gamma_{a'b'c'd'}\ne 0$.
	For any $m\in\Z$, we compute
	\begin{align}
	0= &I_m\phi(1\otimes 1)-\phi(I_m(1\otimes 1))  \notag\\
	=&\sum_{i=0}^a\sum_{j=0}^b\sum_{k=0}^c\sum_{l=0}^d([(\lambda_1')^m\sigma_1'\gamma_{ijkl}(X-1)^i(Y-m)^j-\lambda_1^m\sigma_1\gamma_{ijkl}X^iY^j]\otimes X_1^kY_1^l)  \notag    \\
	&+\sum_{i=0}^a\sum_{j=0}^b\sum_{k=0}^c\sum_{l=0}^d(\gamma_{ijkl}X^iY^j\otimes [(\lambda_2')^m\sigma_2'(X_1-1)^k(Y_1-m)^l-\lambda_2^m\sigma_2X_1^kY_1^l]) \notag\\
	=&[(\lambda_1')^m\sigma_1'-\lambda_1^m\sigma_1+(\lambda_2')^m\sigma_2'-\lambda_2^m\sigma_2]\gamma_{p'q's't'}X^{p'}Y^{q'}\otimes X_1^{s'}Y_1^{t'} + w', \label{iso-Im-1} 
	\end{align}
	where $w'\in\Omega(\lambda_1',\eta_1',\sigma_1',0)\otimes\Omega(\lambda_2',\eta_2',\sigma_2',0)$ and 
	$$ w'=0, {\rm \ or\ } (p',q',s',t')\succ\deg(w').$$
	Taking $m=0,1,2,3$ in Eq.~\eqref{iso-Im-1}, we get
	\begin{equation}\label{44}
	\begin{pmatrix}
	1&-1&1&-1\\            
	\lambda_1&-\lambda_1'&\lambda_2&-\lambda_2'\\
	\lambda_1^2&-(\lambda_1')^2&\lambda_2^2&-(\lambda_2')^2\\
   \lambda_1^3&-(\lambda_1')^3&\lambda_2^3&-(\lambda_2')^3\\
	\end{pmatrix}
		\begin{pmatrix}
\sigma_1\\
\sigma_1'\\
\sigma_2\\
\sigma_2'            
	\end{pmatrix}=
\begin{pmatrix}
0\\
0\\
0\\
0            
\end{pmatrix}.
	\end{equation}
Thus the determinant of the matrix
\[
\begin{pmatrix}
1&-1&1&-1\\            
\lambda_1&-\lambda_1'&\lambda_2&-\lambda_2'\\
\lambda_1^2&-(\lambda_1')^2&\lambda_2^2&-(\lambda_2')^2\\
\lambda_1^3&-(\lambda_1')^3&\lambda_2^3&-(\lambda_2')^3\\
\end{pmatrix}
\]
is zero since $\sigma_1, \sigma_1',\sigma_2,\sigma_2'\in\C^*$.	
By the simple computations we obtain
$$(\lambda_1-\lambda_1')(\lambda_2-\lambda_1)(\lambda_2-\lambda_1')(\lambda_1-\lambda_2')(\lambda_2'-\lambda_1')(\lambda_2'-\lambda_2)=0,$$
which yields $$\lambda_1=\lambda_1', {\rm \ \ or\ \ }
\lambda_1=\lambda_2',{\rm \ \ or\ \ }\lambda_2=\lambda_1',{\rm \ \ or\ \ }\lambda_2=\lambda_2',$$
since $\lambda_1\ne \lambda_2,\lambda_1'\ne \lambda_2'$.
Without loss of generality	we can assume that (otherwise, we can replace $\Omega(\lambda_1,\eta_1,\sigma_1,0)\otimes\Omega(\lambda_2,\eta_2,\sigma_2,0)$
by 
$\Omega(\lambda_2,\eta_2,\sigma_2,0)\otimes\Omega(\lambda_1,\eta_1,\sigma_1,0)$)
$$\lambda_1=\lambda_1', {\rm \ \ or\ \ }\lambda_1=\lambda_2'.$$

\begin{case}\label{11'-1}
	$\lambda_1=\lambda_1'$.
\end{case}	
Then by Eq.~\eqref{44} we get 
\begin{equation}
\begin{pmatrix}
1&1&-1\\            
\lambda_1&\lambda_2&-\lambda_2'\\
\lambda_1^2&\lambda_2^2&-(\lambda_2')^2
\end{pmatrix}
\begin{pmatrix}
\sigma_1-\sigma_1'\\
\sigma_2\\
\sigma_2'            
\end{pmatrix}=
\begin{pmatrix}
0\\
0\\
0            
\end{pmatrix},
\end{equation}
which implies 
$$ \lambda_1=\lambda_2', {\rm \ \ or\ \ }\lambda_2=\lambda_2',$$
since $\sigma_2, \sigma_2'\in\C^*$ and $\lambda_1\ne \lambda_2$.
Moreover, if $\lambda_1=\lambda_2'$, then $$\lambda_1'=\lambda_1=\lambda_2',$$
which is a contradiction. So $\lambda_2=\lambda_2'$.
Using $\lambda_1=\lambda_1'$ and $\lambda_2=\lambda_2'$ we can easily get that
$$\sigma_1=\sigma_1' {\rm  \ \ and\ \ } \sigma_2=\sigma_2',$$
from Eq.~\eqref{iso-Im-1}.

Now substituting $$\lambda_1=\lambda_1',\ \lambda_2=\lambda_2',\ \sigma_1=\sigma_1'{\rm\ \ \and\ \ } \sigma_2=\sigma_2',$$
into Eq.~\eqref{iso-Im-1} we deduce that
for any $m\in\Z$, 
\begin{align}
0= &I_m\phi(1\otimes 1)-\phi(I_m(1\otimes 1))  \notag\\
=&\sum_{i=0}^a\sum_{j=0}^b\sum_{k=0}^c\sum_{l=0}^d(\lambda_1^m\sigma_1\gamma_{ijkl}[(X-1)^i(Y-m)^j-X^iY^j]\otimes X_1^kY_1^l)   \notag   \\
&+\sum_{i=0}^a\sum_{j=0}^b\sum_{k=0}^c\sum_{l=0}^d(\lambda_2^m\sigma_2\gamma_{ijkl}X^iY^j\otimes [(X_1-1)^k(Y_1-m)^l-X_1^kY_1^l]). \label{iso-Im-1-1}
\end{align}
Then by similar disscussions to Theorem \ref{iso-12} we can deduce that 
there exists $\gamma\in\C^*$ such that
\begin{align*}
&\phi(1\otimes 1)=\gamma(1\otimes 1),\ \phi(X\otimes 1)=\gamma(X\otimes 1),\ \phi(1\otimes X_1)=\gamma(1\otimes X_1),\\
&\phi(Y\otimes 1)=\gamma (Y\otimes 1),\ \phi(1\otimes Y_1)=\gamma(1\otimes Y_1),\ \eta_1=\eta_1'\ {\rm \ and \ }\ \eta_2=\eta_2'.
\end{align*}
Hence we have
$$(\lambda_1,\eta_1,\sigma_1)=(\lambda_1',\eta_1',\sigma_1'){\rm\ \  and \ \ } (\lambda_2,\eta_2,\sigma_2)=(\lambda_2',\eta_2',\sigma_2').$$

\begin{case}
	$\lambda_1=\lambda_2'$.
\end{case}
Replacing $$\Omega(\lambda_1',\eta_1',\sigma_1',0)\otimes\Omega(\lambda_2',\eta_2',\sigma_2',0)$$
by 
$$\Omega(\lambda_2',\eta_2',\sigma_2',0)\otimes\Omega(\lambda_1',\eta_1',\sigma_1',0).$$
Then from the processings of Case \ref{11'-1} we can deduce that 
$$(\lambda_1,\eta_1,\sigma_1)=(\lambda_2',\eta_2',\sigma_2'){\rm\ \  and \ \ } (\lambda_2,\eta_2,\sigma_2)=(\lambda_1',\eta_1',\sigma_1').$$
This completes the proof of the theorem.
\end{proof}

By similar proof to Theorem \ref{iso-11} we can get the following theorem.
\begin{theorem}\label{iso-22}
	Let $\lambda_1, \lambda_2,\lambda_1', \lambda_2', \sigma_1, \sigma_2, \sigma_1', \sigma_2'\in\C^*$ and $\eta_1,\eta_2,\eta_1',\eta_2'\in\C$ with $\lambda_1\ne\lambda_2, \lambda_1'\ne\lambda_2'$.
	Then the irreducible $\GG$-modules $$\Omega(\lambda_1,\eta_1,0,\sigma_1)\otimes\Omega(\lambda_2,\eta_2,0,\sigma_2){\rm \ \  and\ \ } \Omega(\lambda_1',\eta_1',0,\sigma_1')\otimes\Omega(\lambda_2',\eta_2',0,\sigma_2')$$
	are isomorphic if and only if 
	$$(\lambda_1,\eta_1,\sigma_1)=(\lambda_1',\eta_1',\sigma_1'), \ \ (\lambda_2,\eta_2,\sigma_2)=(\lambda_2',\eta_2',\sigma_2'),$$
	or 
	$$(\lambda_1,\eta_1,\sigma_1)=(\lambda_2',\eta_2',\sigma_2'), \ \ (\lambda_2,\eta_2,\sigma_2)=(\lambda_1',\eta_1',\sigma_1').$$
\end{theorem}

\section{Applications}
Inspried by Section \ref{Sec3} and Section \ref{Sec4} we obtain some results about tensor product modules over the Witt algebra $\WW$ and the Heisenberg-Virasoro algebra $\LL$.

{\bf Witt algebra}:
Recall that for any $\lambda\in\C^*,\al\in\C$,
 the polynomial algebra $\C[Y]$ has a $\WW$-module structure with the following actions
$$L_m(f(Y))=\lambda^m(Y+m\al)f(Y-m), \ \ \ {\rm \ \ for\ \ all\ \ } m\in\Z,\ f(Y)\in\C[Y].$$
Denote this module by $\Omega(\lambda, \al)$.
Thanks to \cite{LZ2}, we know that $\Omega(\lambda, \al)$ is irreducible if and only if $\al\ne 0$.
Moreover, the \cite{TZ} indicates that $\{\Omega(\lambda, \al)~|~\lambda\in\C^*, \al\in\C\}$ exhaust all $\UU(\C L_0)$-free module of rank one over $\WW$ up to isomorphism.
By Proposition \ref{lambda1=2} we can easily get the following results.

	$\bullet$ Let $\lambda,\al_1,\al_2\in\C^*$. Suppose that $\Omega(\lambda,\al_1)\otimes\Omega(\lambda,\al_2)$ is the tensor product module over $\WW$.
	Denote 
	\begin{align}\label{U}
	U=\mathrm{span}\{\sum_{t=0}^j{\mathrm C}_j^tY^{j-t}\otimes Y^t~|~j\in\N\}.
	\end{align}
	Then $U$ is a nonzero proper submodule of $\Omega(\lambda,\al_1)\otimes\Omega(\lambda,\al_2)$.
	Consequently, $\Omega(\lambda,\al_1)\otimes\Omega(\lambda,\al_2)$ is a reducible $\WW$-module.

Therefore we obtain a necessary condition for the tensor product module $\Omega(\lambda_1,\al_1)\otimes\Omega(\lambda_2,\al_2)$ over $\WW$ to be irreducible, i.e.,
if the tensor product module $\Omega(\lambda_1,\al_1)\otimes\Omega(\lambda_2,\al_2)$ over $\WW$ is irreducible, then $\lambda_1\ne\lambda_2$.

{\bf Heisenberg-Virasoro algebra}:
For $\lambda\in\C^*,\al,\be\in\C$,
it is easy to see that the polynomial algebra $\C[Y]$ is an $\LL$-module with the following actions
$$L_m(f(Y))=\lambda^m(Y+m\al)f(Y-m),\ \ H_m(f(Y))=\be\lambda^mf(Y-m), \ {\rm \ \ for\ \ all\ \ } m\in\Z,\ f(Y)\in\C[Y].$$
We denote by $\Omega(\lambda, \al, \be)$ this module.
From \cite{CG} we know that $\Omega(\lambda, \al,\be)$ is an irreducible $\LL$-module if and only if $(\al,\be)\ne (0,0)$.
Furthermore, $\{\Omega(\lambda, \al, \be)~|~\lambda\in\C^*, \al,\be\in\C\}$ exhaust all $\UU(\C L_0)$-free module of rank one over $\LL$ up to isomorphism.
From Proposition \ref{lambda1=2} the following conclusions are clear.

	$\bullet$ Let $\lambda\in\C^*,\al_1,\al_2,\be_1,\be_2\in\C$ with $(\al_1,\be_1)\ne(0,0), (\al_2,\be_2)\ne(0,0)$.
Suppose that $\Omega(\lambda,\al_1,\be_1)\otimes\Omega(\lambda,\al_2,\be_2)$ is the tensor product module over $\LL$.
Then $U$ defined by Eq.~\eqref{U} is a nonzero proper submodule of $\Omega(\lambda,\al_1,\be_1)\otimes\Omega(\lambda,\al_2,\be_2)$.
Consequently, $\Omega(\lambda,\al_1,\be_1)\otimes\Omega(\lambda,\al_2,\be_2)$ is a reducible $\LL$-module.

Now we give a necessary condition for the tensor product module $\Omega(\lambda_1,\al_1,\be_1)\otimes\Omega(\lambda_2,\al_2,\be_2)$ over $\LL$ to be irreducible, i.e.,
if the tensor product module $\Omega(\lambda_1,\al_1,\be_1)\otimes\Omega(\lambda_2,\al_2,\be_2)$ over $\LL$ is irreducible, then $\lambda_1\ne\lambda_2$.

 {\bf Acknowledgments:}
 This research work is partially supported by NSF of China (Grants 119310
 \newline09, 12071150, 12271265, 12401032).
 J. Cheng is partially supported by Shandong Provincial Natural Science Foundation of China (Grant No. ZR2022MA072) and undergraduate education reform project of Shandong Normal University (Grant No. 2021BJ054).
 D. Gao is partially supported by China National Postdoctoral Program for Innovative Talents (Grant No. BX20220158), 
 China Postdoctoral Science Foundation (Grant No. 2022M711708) and Nankai Zhide Foundation.


\begin{thebibliography}{99}
\bibitem{A}  N. Aizawa, Some properties of planar Galilean conformal algebras, Lie Theory Appl. Phys. 36 (2013), 301-309.

\bibitem{ADKP} E. Arbarello, C. De Concini, V. G. Kac and C. Procesi, Moduli spaces
of curves and representation theory, Commun. Math. Phys. 117 (1988), 1-36.

\bibitem{AP} D. Arnal and G. Pinczon, On algebraically irreducible representations of the Lie algebra $\mathrm{sl}(2)$, J. Math. Phys. 15 (1974), 350-359.

\bibitem{As} A. Astashkevich, On the structure of Verma modules over Virasoro and Neveu-Schwarz algebras, Commun. Math. Phys. 186 (1997), 531-562.

\bibitem{BBM} A. Bagchi, R. Basu and A. Mehra, Galilean conformal electrodynamics,  J. High Energy Phys. 11 (2014), 061.

\bibitem{BG} A. Bagchi and R. Gopakumar, Galilean conformal algebras and AdS/CFT, J. High Energy Phys. 07 (2009), 037.

\bibitem{BBFK} V. Bekkert, G. Benkart, V. Futorny and I. Kashuba, New irreducible modules for Heisenberg and affine Lie algebras, J. Algebra 373 (2013), 284-298.

\bibitem{BMW} S. Bhattacharyya, S. Minwalla and S. R. Wadia, The incompressible non-relativistic Navier-Stokes equation from gravity, J. High Energy Phys. 08 (2009), 059.

\bibitem{BF} Y. Billig and V. Futorny, Classification of irreducible representations of Lie algebra of vector fields on a torus, J. Reine Angew. Math. 720 (2016), 199-216.

\bibitem{CLW} Y. Cai, R. L{\" u} and Y. Wang, Classification of simple Harish-Chandra modules for map (super)algebras related to the Virasoro algebra, J. Algebra 570 (2021), 397-415.

\bibitem{Car} R. Carter, Lie Algebras of Finite and Affine Type, vol. 96 of Cambridge Studies in Advanced Mathematics, Cambridge University Press, 2005.

\bibitem{CP} V. Chari and A. Pressley, A new family of irreducible, integrable modules for affine Lie algebras, Math. Ann. 277 (1987), 543-562.

\bibitem{CG} H. Chen and X. Guo, Non-weight modules over the Heisenberg-Virasoro algebra and the W algebra W(2,2), J. Algebra Appl. 16 (2017), 1750097.

\bibitem{ChenGZ} H. Chen, X. Guo and K. Zhao, Tensor product weight modules over the Virasoro algebra, J. London. Math. Soc. 88 (2013), 829-844.

\bibitem{CY} Q. Chen and Y. Yao, Simple restricted modules for the universal central extension of the planar Galilean conformal algebra, J. Algebra 634 (2023), 698-721.

\bibitem{CYY} Q. Chen, Y. Yao and H. Yang,  Whittaker modules for the planar Galilean conformal algebra and its central extension, Comm. Algebra 50 (2022), 5042-5059.

\bibitem{CGZ} J. Cheng, D. Gao and Z. Zeng, Modules over the planar Galilean conformal algebra arising from free modules of rank one, Pacific J. Math. 326 (2023), 227-249.

\bibitem{FHHO} G. Festuccia, D. Hansen, J. Hartong and N. Obers, Symmetries and couplings of non-relativistic electrodynamics, J. High Energy Phys. 11 (2016), 037.

\bibitem{G} D. Gao, Irreducible modules over the universal central extension of the planar Galilean conformal algebra, submitted.

\bibitem{GG} D. Gao and Y. Gao, Representations of the planar Galilean conformal algebra, Commun. Math. Phys. 391 (2022), 199-221.

\bibitem{GaoZ} D. Gao and K. Zhao, Tensor product weight modules for the mirror Heisenberg-Virasoro algebra, J. Pure Appl. Algebra 226 (2022), 106929.

\bibitem{GZ} X. Guo and K. Zhao, Irreducible representations of non-twisted affine Kac-Moody algebras, arXiv:1305.4059.

\bibitem{Kac} V. Kac, Infinite-Dimensional Lie Algebras, 3rd edn, Cambridge University Press, 1990.

\bibitem{K} B. Kostant, On Whittaker vectors and representation theory, Invent. Math. 48 (1978), 101-184.

\bibitem{LG} G. Liu and X. Guo, Harish-Chandra modules over generalized Heisenberg-Virasoro algebras,
Israel J. Math. 204 (2014), 447-468.

\bibitem{LZ} R. Lu and K. Zhao, Classification of irreducible weight modules over higher rank Virasoro algebras, Adv. Math. 206 (2006), 630-656. 

\bibitem{LZ2} R. Lu and K. Zhao, Irreducible Virasoro modules from irreducible Weyl modules, J. Algebra 414 (2014), 271-287.

\bibitem{LZ1}  R. L{\"u} and K. Zhao, Classification of irreducible weight modules over the twisted Heisenberg-Virasoro algebra, Commun. Contemp. Math. 12 (2010), 183-205.

\bibitem{Ma} J. Maldacena, The large N limit of superconformal field theories and supergravity, Adv. Theor. Math. Phys. 2 (1998), 231-252.
[Int. J. Theor. Phys. 38 (1999), 1113-1133].

\bibitem{MT} D. Martelli and Y. Tachikawa, Comments on Galilean conformal field theories and their geometric realization, J. High Energy Phys. 05 (2010), 091.

\bibitem{M} O. Mathieu, Classification of Harish-Chandra modules over the Virasoro Lie algebra, Invent. Math. 107 (1992), 225-234.


\bibitem {MZ} V. Mazorchuk and K. Zhao, Classification of simple weight Virasoro modules with a finite dimensional weight space, J. Algebra 307 (2007), 209-214.


\bibitem{Nil} J. Nilsson, Simple $\mathrm{sl_{n+1}}$-module structures on $\UU(\mathfrak{h})$, J. Algebra 424 (2015), 294-329.

\bibitem{S} Y. Su, A classification of indecomposable $\mathrm{sl_2(\C)}$-modules and a conjecture of Kac on irreducible modules over the Virasoro algebra, J. Algebra 161 (1993), 33-46.

\bibitem{TZ} H. Tan and K. Zhao, $W^+_n$ and $W_n$-module structures on $\UU(\mathfrak{h_n})$, J. Algebra 424 (2015), 357-375.

\bibitem{Vir} M. A. Virasoro, Subsidiary conditions and ghosts in dual resonance models, Phys. Rev. D (1970), 2933-2936.

\bibitem{ZD} W. Zhang and C. Dong, $W$-algebra $W(2,2)$ and the vertex operator algebra $L(\frac{1}{2},0)\otimes L(\frac{1}{2},0)$, Commun. Math. Phys. 285 (2009), 991-1004.




\end{thebibliography}
\end{document}